\documentclass{amsproc}
\usepackage{amsmath}
\usepackage{amsfonts}
\usepackage{amssymb}
\usepackage{latexsym}
\usepackage{epsfig}
\usepackage{graphicx}
\usepackage{oldgerm}

\usepackage{color}

\def\N{\mathbb{N}}
\def\C{\mathbb{C}}

\def\cO{\mathcal{O}}

\newcommand{\ta}{\theta}
\newcommand{\E}{\mathbb E}

\newtheorem{theorem}{Theorem}[section]
\newtheorem{lemma}[theorem]{Lemma}
\newtheorem{proposition}[theorem]{Proposition}
\newtheorem{corollary}[theorem]{Corollary}

\theoremstyle{definition}
\newtheorem{definition}[theorem]{Definition}

\theoremstyle{remark}
\newtheorem{remark}[theorem]{Remark}

\numberwithin{equation}{section}

\allowdisplaybreaks

\begin{document}

\title[Extensions of the Chaundy--Bullard identity]
{On an identity of Chaundy and Bullard.\\
  III. Basic and elliptic extensions}
\author{N.~Hoshi}
\address{Department of Physics,
Faculty of Science and Engineering,
Chuo University, 
Kasuga, Bunkyo-ku, Tokyo 112-8551, Japan}
\email{hoshi@phys.chuo-u.ac.jp}

\author{M.~Katori}
\address{Department of Physics,
Faculty of Science and Engineering,
Chuo University,
Kasuga, Bunkyo-ku, Tokyo 112-8551, Japan}
\email{katori@phys.chuo-u.ac.jp}
\thanks{The first two authors' research was
supported by the Grant-in-Aid for Scientific Research
(C)(No.19K03674), (B)(No.23H01077), (A)(No.21H04432), 
and the Grand-in-Aid for Transformative Research Areas 
(A)(No.22H05105) 
of the Japan Society for the Promotion of Science.}

\author{T.~H.~Koornwinder}
\address{Korteweg--de Vries Institute,
University of Amsterdam,
P. O. Box 94248, 
1090 GE Amsterdam, The Netherlands}
\email{T.H.Koornwinder@uva.nl}

\author{M.~J.~Schlosser}
\address{Fakult\"{a}t f\"{u}r Mathematik,
  Universit\"{a}t Wien,
Oskar-Morgenstern-Platz 1, A-1090 Wien, Austria}
\email{michael.schlosser@univie.ac.at}

\thanks{The fourth author's research was supported by the
Austrian Science Fund, grant
\texttt{10.55776/P32305}.}

\subjclass[2010] {Primary 05A19;
  Secondary 05A10, 05A30, 05C22, 05C81, 11B65, 33D15, 33E05}

\dedicatory{This paper is dedicated to the memory of Richard A.\ Askey}

\keywords {Chaundy--Bullard identity, $q$-analogue, elliptic analogue,
 binomial identities, theta functions, lattice paths, B\'ezout identity}

\begin{abstract}
The identity by Chaundy and Bullard
expresses $1$ as a sum of two
truncated binomial series in one variable 
where the truncations depend on two different
non-negative integers.
We present basic and elliptic extensions of the
Chaundy--Bullard identity.
The most general result, the elliptic extension,
involves, in addition to the nome $p$ and the base $q$,
four independent complex variables.
Our proof uses a suitable weighted lattice path model.
We also show how three of the basic extensions can
be viewed as B\'ezout identities.
Inspired by the lattice path model,
we give a new elliptic extension of the binomial theorem,
taking the form of an identity for elliptic commuting variables.
We further present variants of the homogeneous form of the
identity for $q$-commuting and for elliptic commuting variables.
%
%


\end{abstract}
\maketitle

\section{Introduction} \label{sec:introduction}

In earlier papers by two of the present authors
\cite{KS08,KS13}
the following identity was analyzed in great detail, 
\begin{equation}
1=(1-x)^{n+1} \sum_{k=0}^{m}
\binom{n+k}{k} x^k
+x^{m+1} \sum_{k=0}^n 
\binom{m+k}{k} (1-x)^k,
\label{CB1}
\end{equation}
where $m, n \in \N_0 :=\{0,1,2,\dots\}$
and $x$ is a variable. 
The authors \cite{KS08} originally attributed this formula to
Chaundy and Bullard \cite{CB60}. As the present paper
constitutes a continuation of the papers \cite{KS08,KS13},
we shall keep referring to \eqref{CB1} as the
\textit{Chaundy--Bullard identity}.
(In the follow-up paper \cite{KS13} to \cite{KS08}
the authors pointed out that an identity
equivalent to \eqref{CB1} was already published by
Pierre Raymond de Montmort in 1713 \cite{M1713},
in one of the very first treatises on probability theory.)
This fundamental formula, expressing $1$ 
as a sum of two truncated binomial series, was rediscovered
many times over more than three hundred years. 
A lot of history and occurrences of this identity in
various areas of mathematics were surveyed in \cite{KS08,KS13},
where several different proofs were compiled
and connections to special functions including
Gau{\ss} hypergeometric series, incomplete beta
integrals and Krawtchouk polynomials were made
explicit\footnote{After seeing \cite{KS08,KS13},
  Slobodan Damjanovic kindly brought to the authors' attention
  that almost at the same time as Chaundy and Bullard
  published their paper \cite{CB60}, Kesava~Menon gave the identity
  \eqref{CB1} in \cite[Equation~(1.2)]{KM61} as well.
  His proof uses partial fraction decomposition of $x^{-m}(1-x)^{-n}$.}.
One of the topics discussed in \cite{KS08} concerns
a multivariate extension of the Chaundy--Bullard
identity related to Lauricella hypergeometric functions,
with an explicit description of the corresponding system of
partial differential equations they satisfy.

One of the main purposes of the present paper is to present
an \textit{elliptic extension of the Chaundy--Bullard identity}.
Its proof (using a suitable lattice path model)
leads us even to discover a new elliptic extension of the binomial theorem.

In order to formulate our theorems,
we introduce some notations. 
First we fix $q \in \C$,
which we call a \textit{base}, and 
for $x \in \C$ and $k \in \N_0$ define
the \textit{$q$-shifted factorials} $(x;q)_k$
(also known as \textit{$q$-Pochhammer symbols}) by
\begin{equation}
  (x; q)_0:= 1\quad\;\text{and}\quad(x; q)_k:=
  \prod_{\ell=0}^{k-1}(1-x q^{\ell}),\quad k=1,2,\dots\;.
\label{q_Poch1}
\end{equation}
For $|q|<1$ we may also take $k=\infty$ in \eqref{q_Poch1};
the $q$-shifted factorial is then a convergent infinite product.
For $x_1, \dots, x_s \in \C$ 
products of $q$-shifted factorials are abbreviated as
$(x_1, \dots, x_s; q)_k
:=\prod_{i=1}^s (x_i; q)_k$, where $x_1, \dots, x_s \in \C$ and
$k \in \N_0\cup\{\infty\}$.
Next we fix another parameter $p \in \C$ with $0<|p| <1$.
The \textit{modified Jacobi theta function} with argument 
$x \in \C^{\times} := \C \setminus \{0\}$ and \textit{nome} $p$ 
is defined by
\begin{equation}
\theta(x;p)
:= (x, p/x; p)_{\infty}.
\label{theta1}
\end{equation}
We will frequently abbreviate products of modified Jacobi
theta functions using the notation
$\theta(x_1, \dots, x_s; p)
=\prod_{i=1}^s \theta(x_i; p)$ for
$x_1, \dots, x_s \in \C^{\times}$.
For $x \in \C^{\times}$ and  $k \in \N_0$,
the \textit{theta-shifted factorial} $(x;q,p)_k$ (also called the
\textit{$q,p$-shifted factorial}) is defined by
\begin{equation*}
(x; q, p)_0:=1\quad\;\text{and}\quad(x; q, p)_k:=
\prod_{\ell=0}^{k-1} \theta(x q^{\ell}; p),\quad k=1,2,\dots\;.
\end{equation*}
A product of theta-shifted factorials
is compactly written as
$(x_1, \dots, x_s; q,p)_k
:= \prod_{i=1}^s (x_i; q, p)_k$,
where $x_1, \dots, x_s \in \C^{\times}$, $k \in \N_0$.

Three simple identities satisfied by the modified
Jacobi theta function are the symmetry
$$
\theta(x;p)=\theta(p/x;p),
$$
the inversion formula
$$
\theta(1/x;p)=-\frac 1x\theta(x;p),
$$
and the quasi-periodicity
$$
\theta(px;p)=-\frac 1x\theta(x;p),
$$
which all easily follow from the definition in \eqref{theta1}.
In this paper we will make crucial use of the following identity, the 
\textit{Weierstra{\ss}--Riemann addition formula}
\cite[p.~451, Example~5]{WW62}
\begin{equation}
\theta(xy,x/y,uv,u/v;p)-\theta(xv,x/v,uy,u/y;p)
=\frac{u}{y}\theta(yv,y/v,xu,x/u;p),
\label{Weierstrass_Riemann}
\end{equation}
where $x, y, u, v \in \C^{\times}$. For a discussion of
\eqref{Weierstrass_Riemann} and a comparison with similar
relations, see \cite{Koo14}.

Whereas we refer to the extensions of the Chaundy--Bullard identity
that involve $q$-shifted factorials as \textit{$q$-extensions}
or \textit{basic extensions} (with $q$ being the base)
of the Chaundy--Bullard identity,
our \textit{elliptic extension} of the Chaundy--Bullard identity
connects series containing ratios of products
of modified Jacobi theta functions
(these ratios are in fact \textit{elliptic functions}, justifying
the terminology). These basic and elliptic extensions
actually involve truncated basic and elliptic hypergeometric series
(the text book \cite{GR04} contains a comprehensive treatise of the
theory of these series; the elliptic cases is treated in Chapter 11).
In Section~\ref{sec:proof} all our extensions of the Chaundy--Bullard
identity are obtained from scratch,
without requiring results from the theories of
basic or elliptic hypergeometric series.

By definition, a function $g(u)$ is {\em elliptic}, if it is
a doubly-periodic meromorphic function of the complex variable $u$.

As a consequence of the theory of theta functions
(cf.\ \cite[Theorem~1.3.3]{R21}) one may assume without
loss of generality that
\begin{equation*}
g(u)=\frac{\ta(a_1q^u,a_2q^u,\dots,a_sq^u;p)}
{\ta(b_1q^u,b_2q^u,\dots,b_sq^u;p)}\,z
\end{equation*}
for a positive integer $s$, a constant $z$ and some
$a_1,a_2,\dots,a_s$, $b_1,\dots,b_s$, and $p,q$ with $|p|<1$,
where the {\em elliptic balancing condition} (cf.\ \cite{Sp02}), namely
\begin{equation*}
a_1a_2\cdots a_s=b_1b_2\cdots b_s,
\end{equation*}
holds. (So, a linear combintation of two such expressions is again
an expression of the same form.)
If one writes $q=e^{2\pi\sqrt{-1}\sigma}$, $p=e^{2\pi\sqrt{-1}\tau}$,
with complex $\sigma$ and $\tau$ and $\Im\tau>0$,
then $g(u)$ is indeed periodic in $u$
with periods $\sigma^{-1}$ and $\tau\sigma^{-1}$.
Keeping this notation for $p$ and $q$,
denote the {\em field of elliptic functions} over $\C$
of the complex variable $u$, meromorphic in $u$
with the two periods $\sigma^{-1}$ and $\tau\sigma^{-1}$ by
$\E_{q^u;q,p}$.

More generally, denote the {\em field of totally elliptic multivariate
functions} over $\C$ of the complex variables $u_1,\dots,u_n$,
meromorphic in each variable\footnote{Here we would like to mention
Hartogs' theorem and its analogue for meromorphic functions.
Hartogs' theorem says informally that a multivariate function that
is separably analytic (i.e., analytic in each independent variable)
is analytic. Next, a separably meromorphic function is
meromorphic, see \cite[Corollary~2]{Sh86}.}
with equal periods,
$\sigma^{-1}$ and $\tau\sigma^{-1}$, of double periodicity, by
$\E_{q^{u_1},\dots,q^{u_n};q,p}$.
The notion of totally elliptic multivariate functions
was first introduced by Spiridonov, see
\cite[p.~317, Definition~11]{Sp02} (where the related
notion of totally elliptic hypergeometric series was defined)
and \cite[Definition~6]{Sp11}.

We are ready to state our elliptic extension of the
Chaundy--Bullard identity.

\begin{theorem}
\label{thm:elliptic_CB}
Let $a, b, c, x, q \in \C^{\times}$, $p\in\C$ with $0<|p|<1$
and $n,m\in\N_0$. Then, as an identity in $\E_{x,a,b,c;q,p}$,
\begin{align}
\label{CB_elliptic1}
  1 &=\frac{(ac, c/a, bx, b/x; q, p)_{n+1}}{(ab, b/a, cx, c/x; q, p)_{n+1}}\\*
  &\quad\;\times
\sum_{k=0}^{m}
\frac{\theta(acq^{n+2k}; p)
(acq^n, bcq^n, c/b, q^{n+1}, a x, a/x; q, p)_k}
{\theta(acq^n; p)
(q, aq/b, abq^{1+n}, ac, cq^{n+1}/x, cxq^{n+1}; q, p)_k} 
q^k
\nonumber\\
    &\;+ \frac{(bc, c/b, ax, a/x; q, p)_{m+1}}
      {(ab, a/b, cx, c/x; q, p)_{m+1}}
      \nonumber\\*
  &\quad\;\times
\sum_{k=0}^n 
\frac{\theta(bcq^{m+2k}; p)
(bcq^m, acq^m, c/a, q^{m+1}, b x, b/x; q, p)_k}
{\theta(bcq^m; p)
(q, bq/a, ab q^{1+m}, bc, cq^{m+1}/x, cxq^{m+1}; q, p)_k} 
q^k.\nonumber
\end{align}
\end{theorem}

In order to demonstrate that the identity \eqref{CB_elliptic1}
is an extension of the original Chaundy--Bullard
identity \eqref{CB1}, we 
show how \eqref{CB_elliptic1} can be reduced to \eqref{CB1}
by taking suitable limits.
Since our formula \eqref{CB_elliptic1} involves
three parameters $a,b,c \in \C^{\times}$
in addition to the variable $x$, the base $q$ and the nome $p$, 
we have several intermediate identities between
\eqref{CB1} and \eqref{CB_elliptic1}.
Let
\begin{align}\label{CB_elliptic0}
  &p_{m,n}(x; a,b,c; q, p):=
  \frac{(ac, c/a, bx, b/x; q, p)_{n+1}}{(ab, b/a, cx, c/x; q, p)_{n+1}}\\*
 \nonumber &\times
\sum_{k=0}^{m}
\frac{\theta(acq^{n+2k}; p)
(acq^n, bcq^n, c/b, q^{n+1}, a x, a/x; q, p)_k}
{\theta(acq^n; p)
(q, aq/b, abq^{1+n}, ac, cq^{n+1}/x, cxq^{n+1}; q, p)_k} 
q^k.
\end{align}
Then \eqref{CB_elliptic1} can be written as
\begin{equation}
1= p_{m, n}(x; a,b,c; q, p) + p_{n, m}(x; b,a,c;q,p). 
\label{CB_elliptic2}
\end{equation}
By the definition of the modified Jacobi theta function \eqref{theta1}
with \eqref{q_Poch1}, one has
$\lim_{p \to 0} \theta(x;p)=1-x$, 
and hence $\lim_{p \to 0} (x; q, p)_k=(x;q)_k$, $k \in \N_0$; 
that is the theta-shifted factorial is reduced to
the $q$-shifted factorial in the limit $p \to 0$.
Then the identity \eqref{CB_elliptic2} with $p=0$ holds with
\begin{align}\label{pmn_abcq}
&p_{m,n}(x; a,b,c;q,0):=
\frac{(ac, c/a, bx, b/x; q)_{n+1}}{(ab, b/a, cx, c/x; q)_{n+1}}\\*\nonumber
&\times\sum_{k=0}^{m}
\frac{(1-acq^{n+2k})
(acq^n, bcq^n, c/b, q^{n+1}, a x, a/x; q)_k}
{(1-acq^n)
(q, aq/b, abq^{1+n}, ac, cq^{n+1}/x, cxq^{n+1}; q)_k} 
q^k.
\end{align}
Next, if in addition to $p\to 0$ we take the further limit $c \to 0$, then 
we have the equality \eqref{CB_elliptic2} with $p=c=0$, in which
\begin{equation}
p_{m,n}(x; a,b,0;q,0)
:=\frac{(bx, b/x; q)_{n+1}}{(ab, b/a; q)_{n+1}} 
\sum_{k=0}^{m}
\frac{(q^{n+1}, a x, a/x; q)_k}{(q, aq/b, abq^{1+n}; q)_k} 
q^k.
\label{pmn_abq_2nd}
\end{equation}
An equivalent form of this $q$-extension of the Chaundy--Bullard
identity was obtained by Ma in \cite[Corollary~4.2]{M09}
as a consequence of his six-variable generalization of Ramanujan's
reciprocity theorem.

We can obtain another variant of a three-parametric
Chaundy--Bullard identity for $x \in \C$ with parameters $a,b,q$.
In \eqref{CB_elliptic2} with $p=c=0$ and \eqref{pmn_abq_2nd}, 
we make the substitution $a \mapsto \delta a$, $b \mapsto b \delta$,
$x \mapsto x/\delta$ and then take the limit
$\delta \to 0$.
The obtained equality is
\begin{equation}
1= \widetilde{p}_{m, n}(x; a, b; q)
+\widetilde{p}_{n,m}(x; b, a; q)
\label{CB_abq_1st}
\end{equation}
with
\begin{equation}
\widetilde{p}_{m,n}(x; a,b;q)
:=\frac{(bx;q)_{n+1}}{(b/a; q)_{n+1}} 
\sum_{k=0}^{m}
\frac{(q^{n+1}, a x; q)_k}{(q, aq/b; q)_k} 
q^k.
\label{pmn_abq_1st}
\end{equation}
In the above identity \eqref{CB_abq_1st} with \eqref{pmn_abq_1st}
the variable $b$ is actually redundant (the substitutions
$a\mapsto ab$ and $x\mapsto x/b$ applied to \eqref{pmn_abq_1st}
eliminate the variable $b$) but it is useful to keep the variable
$b$ for the $a\leftrightarrow b$ symmetry.

Furthermore, if we substitute $x \mapsto x/b$ 
in \eqref{CB_abq_1st} with \eqref{pmn_abq_1st} and then take the
limit $b \to 0$, the equality is reduced to
the identity
\begin{equation}
1=(x; q)_{n+1} \sum_{k=0}^{m}
\begin{bmatrix}n+k\\k \end{bmatrix}_{q}
x^k
+x^{m+1} \sum_{k=0}^n 
\begin{bmatrix}m+k\\k \end{bmatrix}_{q}
q^k (x; q)_k,
\label{CB_q1}
\end{equation}
where the \textit{$q$-binomial coefficient} is defined by
\begin{equation}
\begin{bmatrix}n\\k \end{bmatrix}_{q}
:= \frac{(q;q)_n}{(q;q)_k (q;q)_{n-k}},
\quad k \in \{0, 1, \dots, n \}.
\label{q_binom1}
\end{equation}
Finally, if we take the limit $q \to 1$, then 
\eqref{CB_q1} is reduced to
the original Chaundy--Bullard identity \eqref{CB1}.
We note that the $q \to 1$ limit of
the equality \eqref{CB_abq_1st} 
with \eqref{pmn_abq_1st} is equivalent to
the original Chaundy--Bullard identity \eqref{CB1}
where $x$ is replaced by $(1-ax)/(1-a/b)$,
in light of the easily checked relation
$1-(1-ax)/(1-a/b)=(1-bx)/(1-b/a)$.
To better distinguish the identities, we say that \eqref{CB_q1} is the
\textit{$q$-extension}, 
\eqref{CB_abq_1st} with \eqref{pmn_abq_1st} 
the \textit{$(a,b;q)$-extension of the first kind}, 
\eqref{CB_elliptic2} in the case $p=c=0$ with
\eqref{pmn_abq_2nd} 
the \textit{$(a,b;q)$-extension of the second kind}, 
\eqref{CB_elliptic2} in the case $p=0$ 
with \eqref{pmn_abcq} the \textit{$(a,b,c;q)$-extension},
and \eqref{CB_elliptic1}
the \textit{$(a,b;q,p)$-extension}, or simply
\textit{elliptic extension} of
the Chaundy--Bullard identity, respectively.
Summarizing the above linear scheme, we have the chain of objects

\centerline{\eqref{CB_elliptic0} $\to$ \eqref{pmn_abcq}
  $\to$ \eqref{pmn_abq_2nd} $\to$ \eqref{pmn_abq_1st} $\to$
  \eqref{CB_q1} $\to$ \eqref{CB1}}

\noindent which respectively contain the following free variables:

\centerline{($x,a,b,c;q,p$) $\to$ ($x,a,b,c;q$) $\to$
($x,a,b;q$) $\to$ ($x,a;q$) $\to$ ($x;q$) $\to$ ($x$).}

At this point we would also like to remark that the various $q$-extensions
of the Chaundy--Bullard identity give rise to parameter dependent
Chaundy--Bullard identities after replacing the parameters
$a$, $b$, and $c$, by $q^a$, $q^b$, and $q^c$, respectively,
and taking the limit $q\to 1$. (We leave the details to the reader.)

From the many different proofs that are
known for the original Chaundy--Bullard identity \eqref{CB1}
\cite{KS08,KS13}, in order to prove our elliptic extension
in Theorem~\ref{thm:elliptic_CB},
we develop the fifth proof given in \cite{KS08};
a proof by enumerating weighted lattice paths.
Our construction of a suitable lattice path model
with the proof of Theorem~\ref{thm:elliptic_CB}
is given in Section~\ref{sec:proof}.
In Section \ref{sec:Bezout},  we show how
the $q$-extension of the Chaundy--Bullard identity 
\eqref{CB_q1} and its two kinds of
$(a,b;q)$-extensions can be alternatively derived
by making use of B\'ezout's identity. 
In Section~\ref{sec:q_commuting} we look at variants
of the $q$-extended Chaundy--Bullard identity and
relate them to corresponding identities for
$q$-commuting variables.
Finally, in Section~\ref{sec:elliptic_binomial},
inspired by the lattice path model, we present 
a new elliptic binomial theorem, taking the form of
an identity for elliptic commuting variables.
The elliptic binomial theorem in Theorem~\ref{vwdbinth}
is similar to a result discovered earlier by one of the
authors \cite[Theorem~2]{Sch20} and can be regarded
as a companion result to that other one.
In the same section we also present variants
of the homogeneous form of the Chaundy--Bullard identity
for elliptic commuting variables.

\section{Proof of Theorem \ref{thm:elliptic_CB} 
by enumerating weighted lattice paths} 
\label{sec:proof}
\subsection{General arguments}
\label{sec:general}
Let $m, n \in \N_0$ and consider a rectangular region of the 
square lattice,
\[
\Lambda_{m+1, n+1}
:=\{(i, j) : i \in \{0, 1, \dots, m+1 \},
j \in \{0, 1, \dots, n+1 \} \}.
\]
Let $\Pi_{m+1, n+1}$ be the set of
all lattice paths from $(0,0)$ to $(m+1, n+1)$ in
$\Lambda_{m+1, n+1}$ using only
unit east steps $(i, j) \to (i+1,j)$
and unit north steps $(i, j) \to (i, j+1)$. 
Such a path $\pi$ consists of $m+n+2$ successive
unit steps,
$\pi = \{s_t(\pi), t \in \{1, 2, \dots, m+n+2\} \}$. 
Each step $s$ in $\pi \in \Pi_{m+1, n+1}$ is assigned
a weight $w(s)$.
The weight $w(\pi)$ of a path $\pi \in \Pi_{m+1, n+1}$ 
is defined to be a product of the weights of the respective steps
of the path:
\[
w(\pi):=\prod_{s \in \pi} w(s)
= \prod_{t=1}^{m+n+2} w(s_t(\pi)). 
\]

Next we specify the weights for each of the possible steps of paths
in $\Pi_{m+1, n+1}$.
Let $h:\Lambda_{m,n}\to\C$ be a function which we will specialize later.
For each $i \in \{0,1, \dots, m \}$
we define the weight of a unit east step by
\begin{subequations}\label{weight_general}
\begin{equation}
w((i,j) \to (i+1, j)):= 
\begin{cases}
h(i,j),
& \mbox{if $j \in \{0, 1, \dots, n \}$},
\cr
1,
& \mbox{if $j=n+1$},
\end{cases}
\end{equation}
and for each $j \in \{0,1, \dots, n \}$
the weight of a unit north step by
\begin{equation}
w((i,j) \to (i, j+1)):= 
\begin{cases}
1-h(i,j),
& \mbox{if $i \in \{0, 1, \dots, m \}$},
\cr
1,
& \mbox{if $i=m+1$}. 
\end{cases}
\end{equation}
\end{subequations}
We assume that
\begin{align}
\label{h_condition}
  h(i, 0) \not=0, \quad&\text{for}\quad i \in \{0,1,\dots, m-1\},
\quad\text{and}  \\*\nonumber
h(0, j) \not=1, \quad&\text{for}\quad j \in \{0,1,\dots, n-1\}.
\end{align}
For $\pi \in \Pi_{m+1, n+1}$ and $\tau \in \{0,1, 2, \dots, m+n+2\}$, 
define the truncated path $\pi_{\tau}$ 
by the path $\pi$ terminated after $\tau$ steps.
In particular, $\pi_0=\emptyset$ and $\pi_{m+n+2} = \pi$. 
By the specific choices of weights in \eqref{weight_general}, 
for each $\pi_{\tau-1}$,
$\tau \in \{1,2, \dots, m+n+2\}$, we have
$$
\sum_{\pi_{\tau} : \pi_{\tau} \setminus s_{\tau}(\pi_{\tau})=\pi_{\tau-1}}
w(s_{\tau}(\pi_{\tau}))=1.
$$
Hence by induction we can conclude that for each
$\kappa<\tau$, $\tau \in \{1,2, \dots, m+n+2\}$, we have
$$
\sum_{\pi_{\tau} : \pi_{\tau} \setminus \{s_{\kappa+1}(\pi_{\tau}),
  \ldots,s_{\tau}(\pi_{\tau})\}=\pi_\kappa}
\prod_{i=\kappa+1}^\tau w(s_i(\pi_i))=1.
$$
and by setting $\kappa=0$ obtain
$$
\sum_{\pi_{\tau}: \pi \in\Pi_{m+1, n+1} } w(\pi_{\tau})=1\quad
\text{for all}\quad \tau \in \{1, 2, \dots, m+n+2\}.
$$
In the case $\tau=m+n+2$, the above gives
\begin{equation}
\sum_{\pi \in \Pi_{m+1, n+1}} w(\pi)=1.
\label{sum1}
\end{equation}

For any $(k,\ell)\in\Lambda_{m,n}$, the generating function
$A(k,\ell)$ for all weighted lattice paths
from $(0,0)$ to $(k,\ell)$ is defined by $A(0,0) :=1$ and
\[
A(k, \ell) := \sum_{\pi : (0,0) \to (k, \ell)} w(\pi),
\quad \text{for}\quad (k, \ell) \in \Lambda_{m,n} \setminus \{(0,0)\}, 
\]
where the sum is taken over $\{\pi : (0,0) \to (k, \ell)\}$,
the set of all lattice paths 
from $(0,0)$ to $(k, \ell)$.
By the specific assignment of weights in \eqref{weight_general}, 
$A(k, \ell)$ satisfies the recurrence relation
\begin{subequations}
  \label{A3}
\begin{align}\label{A3a}
&h(k-1, \ell) A(k-1, \ell) 
+(1-h(k, \ell-1)) A(k, \ell-1) =A(k, \ell),\\*
  \nonumber
&\text{for}\quad k \in \{1,2, \dots, m \} \quad
\text{and}\quad \ell \in \{1,2, \dots, n\},
\end{align}
with the boundary conditions
\begin{alignat}{2}\label{A3b}
A(k, 0) &= \prod_{i=0}^{k-1} h(i, 0), 
\quad &&\text{for}\quad k \in \{1, 2, \dots, m\},\\*
\nonumber
A(0, \ell) &= \prod_{j=0}^{\ell-1} (1-h(0, j)),
\quad &&\text{for}\quad \ell \in \{1, 2, \dots, n \}.
\end{alignat}
\end{subequations}
Moreover, since the last step of a path in $\Pi_{m+1,n+1}$
which is not a step along the north or east boundary is either
a step $(k,n)\to(k,n+1)$ ($k=0,1,\ldots,m$) or a step
$(m,\ell)\to(m+1,\ell)$ ($\ell=0,\ldots,n$),
the above assignment of weights implies the equality
\[
\sum_{\pi \in \Pi_{m+1, n+1}} w(\pi)
=\sum_{k=0}^m (1-h(k, n)) A(k,n)  
+\sum_{\ell=0}^n h(m, \ell) A(m, \ell) .
\]
Hence by \eqref{sum1} we have the equality
\begin{equation}
1=\sum_{k=0}^m (1-h(k, n)) A(k,n)  
+\sum_{\ell=0}^n h(m, \ell) A(m, \ell).
\label{equality}
\end{equation}
Under assumption \eqref{h_condition}, 
we put 
\[
B(k, \ell):= \frac{A(k, \ell)}{A(k,0) A(0, \ell)},
\quad\text{for} \quad (k, \ell) \in \Lambda_{m,n}.
\]
Then we have from \eqref{A3} the following system of difference equations
\begin{align}\label{systemB}
&\frac{h(k-1, \ell)}{h(k-1,0)} 
B(k-1, \ell)
+ \frac{1-h(k, \ell-1)}{1-h(0, \ell-1)} 
B(k, \ell-1) =B(k, \ell),\\*
&\text{for}\quad k \in \{1, 2, \dots, m\} \quad\text{and}\quad
\ell \in \{1,2, \dots, n\}, 
\nonumber\\*
& B(k,0) =1, \quad \text{for}\quad k \in \{0,1, \dots, m\},
\nonumber\\*
& B(0, \ell) =1, \quad\text{for}\quad  \ell \in \{0,1, \dots, n\},
\nonumber
\end{align}
which conversely uniquely determines the sequence
$\big(B(k, \ell)\big)_{(k, \ell) \in \Lambda_{m,n}}$.

The above argument is summarized as follows.
\begin{proposition}
\label{thm:main_prop}
Assume that 
$\big(h(i,j)\big)_{(i,j) \in \Lambda_{m,n}}$ is given so that
\eqref{h_condition} is satisfied. 
Let $\big(B(k, \ell)\big)_{(k, \ell) \in \Lambda_{m, n}}$
be uniquely given by \eqref{systemB}. Define
$\big(A(k, \ell)\big)_{(k, \ell) \in \Lambda_{m,n}}$
by $A(0, 0) := 1$, \eqref{A3b} and 
\begin{align}\label{B_to_A}
&A(k, \ell) := A(k, 0) A(0, \ell) B(k, \ell),\\*\nonumber
&\text{for}\quad  k \in \{1,2, \dots, m \}
\quad \text{and}\quad \ell \in \{1,2, \dots, n \}. 
\end{align}
Then the equality \eqref{equality} holds.
\end{proposition}

\subsection{Proof of \eqref{CB_elliptic1}}
\label{sec:proof_elliptic_ext}

We choose the weight function $h(i,j)$ as 
\begin{equation}
h(i,j)= h_{x;a,b,c;q,p}(i, j)
:=\frac{\theta(bcq^{i+2j},(c/b)q^i,ax q^i,(a/x)q^i;p)}
{\theta(abq^{i+j},(a/b) q^{i-j},cxq^{i+j},(c/x)q^{i+j};p)},
\label{h_elliptic1}
\end{equation}
where $a,b,c,x,q,p\in\C^{\times}$ and $|p|<1$.
This choice is motivated by the following symmetry relation.

\begin{lemma}
\label{thm:h_symmetry}
The following equality holds,
\begin{equation}
1- h_{x;a,b,c;q,p}(i, j)= h_{x;b,a,c;q,p}(j,i).
\label{h_symmetry}
\end{equation}
\end{lemma}
\begin{proof}
By \eqref{h_elliptic1}, 
\begin{align*}
&1- h_{x;a,b,c;q,p}(i, j) =
1-\frac{\theta(bcq^{i+2j},(c/b)q^i,ax q^i,(a/x)q^i;p)}
{\theta(abq^{i+j},(a/b) q^{i-j},cxq^{i+j},(c/x)q^{i+j};p)}\\
&=\Big(\theta(abq^{i+j},(b/a) q^{j-i},cxq^{i+j},(c/x)q^{i+j};p)\\*
&\qquad+(b/a)q^{j-i}\theta(bcq^{i+2j},(c/b)q^i,ax q^i,(a/x)q^i;p)\Big)\\*
&\quad\times\theta(abq^{i+j},(b/a) q^{j-i},cxq^{i+j},(c/x)q^{i+j};p)^{-1}.
\end{align*}
Now, with the substitution of variables
$(x,y,u,v)\mapsto(cq^{i+j},aq^i,bq^j,x)$
in the Weierstra{\ss}--Riemann addition formula \eqref{Weierstrass_Riemann},
specifically
\begin{align*}
\theta(abq^{i+j},(b/a) q^{j-i},cxq^{i+j},(c/x)q^{i+j};p)
+\frac{b}{a}q^{j-i}\theta(bcq^{i+2j},(c/b)q^i,ax q^i,(a/x)q^i;p)&\\*
=\theta(acq^{2i+j},(c/a)q^j,bx q^j,(b/x)q^j;p)&,
\end{align*}
the relation 
\[
1- h_{x;a,b,c;q,p}(i, j) 
=\frac{\theta(acq^{2i+j},(c/a)q^j,bx q^j,(b/x)q^j;p)}
{\theta(abq^{i+j},(b/a) q^{j-i},cxq^{i+j},(c/x)q^{i+j};p)}
\]
is established as desired and the proof is complete.
\end{proof}

We would like to explain the motivation for our specific choice of the
weight function \eqref{h_elliptic1}.
Our weighted lattice model has a natural companion obtained by
reflection with respect to the diagonal going northeast form the origin.
Then according to \eqref{weight_general}
the corresponding weights $\tilde w$ of this companion are in terms of,
say, $\tilde h(i,j):=1-h(j,i)$.
So it would be pleasant to have a nice expression for $h$ such that
$\tilde h$ also has a nice expression. Now observe that
\eqref{Weierstrass_Riemann} can be rewritten as
$1-h_{x;a,b,c;q,p}(0,0)=h_{x;b,a,c;q,p}(0,0)$ with $h_{x;a,b,c;q,p}(0,0)$
given by \eqref{h_elliptic1} (which does not yet dependent on $q$).
This gives a motivation for defining
$h_{x;a,b,c;q,p}(i,j):=h_{x;aq^i,bq^j,cq^{i+j};q,p}(0,0)$ so that we have
$\tilde h_{x;a,b,c;q,p}(i,j)=1-h_{x;b,a,c;q,p}(i,j)$.

\smallskip
With $h(i,j)$ as given as in \eqref{h_elliptic1} the system
\eqref{systemB} can be explicitly written as
\begin{align}\label{system_elliptic}
& \frac{\theta(bcq^{k+2\ell-1},abq^{k-1},(a/b) q^{k-1},cxq^{k-1},
(c/x)q^{k-1};p)}
{\theta(abq^{k+\ell-1},(a/b)q^{k-\ell-1},cxq^{k+\ell-1},
(c/x)q^{k+\ell-1},bcq^{k-1};p)} B(k-1, \ell) \\*
&+ \frac{\theta(acq^{2k+\ell-1},abq^{\ell-1},(b/a)q^{\ell-1},
cxq^{\ell-1},(c/x)q^{\ell-1};p)}
{\theta(abq^{k+\ell-1},(b/a)q^{\ell-k-1},cxq^{k+\ell-1},
(c/x)q^{k+\ell-1},acq^{\ell-1};p)}
B(k, \ell-1)\nonumber\\*&=B(k, \ell),
\nonumber\\*
& \text{for}\quad k \in \{1, 2, \dots, m\}\quad\text{and}
\quad \ell \in \{1, 2, \dots, n \},
\nonumber
\\*
& B(k, 0) = 1, 
\quad\text{for}\quad k \in \{0, 1, \dots, m\}, 
\nonumber\\*
& B(0, \ell) = 1,
\quad\text{for}\quad \ell \in \{0, 1, \dots, n \}.
\nonumber
\end{align}

\begin{lemma}
\label{thm:Belliptic}
The unique solution of \eqref{system_elliptic}
is given by
\begin{equation}
B(k, \ell) = 
\frac{\theta((a/b)q^{k-\ell},b/a;p)(bcq^{\ell}; q, p)_k
(acq^k, ab, cx, c/x, q^{k+1}; q, p)_{\ell}}
{\theta((a/b)q^k,(b/a)q^{\ell};p)(bc;q,p)_k
(ac, abq^k, cxq^k, (c/x)q^k, q; q, p)_{\ell}} q^{\ell},
\label{ellipticB}
\end{equation}
for all $(k, \ell) \in \Lambda_{m,n}$.
\end{lemma}
\begin{proof}
The proof proceeds by induction on $k+\ell$. Since $B(0,0)=1$ by the
system \eqref{system_elliptic}, the $k+\ell=0$ case is trivial.
Further, if $k=0$ or $\ell=0$, the values $B(k,0)=B(0,\ell)=1$
specified in \eqref{system_elliptic} agree with those in \eqref{ellipticB}.
Finally, let $k,\ell>0$ and assume that the solution for $B(i,j)$
is given in  \eqref{ellipticB} for all $(i,j)\in\Lambda_{m,n}$
with $i+j<k+\ell$.
Then the left-hand side of \eqref{system_elliptic} is,
with \eqref{ellipticB} applied to rewrite
$B(k-1,\ell)$ and $B(k,\ell-1)$ by induction,
\begin{align*}
&
\frac{\theta(b/a;p)(bcq^{\ell}; q, p)_k
(acq^k, ab, cx, c/x, q^{k+1}; q, p)_{\ell}q^{\ell}}
{\theta((a/b)q^k,(b/a)q^{\ell};p)(bc;q,p)_k
(ac, abq^k, cxq^k, (c/x)q^k, q; q, p)_{\ell}}\\*
&\times \bigg( \frac{\theta(bcq^{k+2\ell-1},acq^{k-1},q^k,(a/b)q^k;p)}
{\theta(acq^{k+\ell-1},bcq^{k+\ell-1},q^{k+\ell};p)}\\*
  &\quad\;\;-\frac{a}{b}q^{k-\ell}
    \frac{\theta(acq^{2k+\ell-1},bcq^{\ell-1},q^{\ell},(b/a)q^{\ell};p)}
{\theta(acq^{k+\ell-1},bcq^{k+\ell-1},q^{k+\ell};p)} \bigg)\\
&=\frac{C(k, \ell)}
{\theta((a/b)q^{k-\ell}, acq^{k+\ell-1}, bcq^{k+\ell-1}, q^{k+\ell}; p)}\\*
 &\quad\times
 \frac{\theta((a/b)q^{k-\ell},b/a;p)(bcq^{\ell}; q, p)_k
(acq^k, ab, cx, c/x, q^{k+1}; q, p)_{\ell}}
{\theta((a/b)q^k,(b/a)q^{\ell};p)(bc;q,p)_k
(ac, abq^k, cxq^k, (c/x)q^k, q; q, p)_{\ell}} q^{\ell},
\end{align*}
with
\begin{align*}
C(k, \ell) &=
             \theta(bcq^{k+2\ell-1},acq^{k-1},q^k,(a/b)q^k;p)\\*
  &\quad-\frac{a}{b}q^{k-\ell}
\theta(acq^{2k+\ell-1},bcq^{\ell-1},q^{\ell},(b/a)q^{\ell};p). 
\end{align*}
Here we again use the Weierstra{\ss}--Riemann addition
formula \eqref{Weierstrass_Riemann}, now with the substitution
$(x,y,u,v)\mapsto(a^{\frac{1}{2}}c^{\frac{1}{2}}q^{k+\ell-\frac{1}{2}},
a^{-\frac{1}{2}}bc^{\frac{1}{2}}q^{\ell-\frac{1}{2}},
a^{\frac{1}{2}}c^{\frac{1}{2}}q^{k-\frac{1}{2}}, 
a^{\frac{1}{2}}c^{\frac{1}{2}}q^{-\frac{1}{2}})$.
This yields the equality
$$C(k, \ell)=
\theta((a/b)q^{k-\ell}, acq^{k+\ell-1}, bcq^{k+\ell-1}, q^{k+\ell}; p)$$
and the proof is complete.
\end{proof}

\begin{proof}[Proof of \eqref{CB_elliptic1}] 
With the specific choice of $h(i,j)$ in \eqref{h_elliptic1} we obtain
from \eqref{A3b} and \eqref{B_to_A} the following explicit
formulas for $A(k,l)$:
\begin{align*}
A(k, 0) &= \prod_{i=0}^{k-1} 
\frac{\theta(bcq^i,(c/b)q^i,axq^i,(a/x)q^i;p)}
{\theta(abq^i,(a/b)q^i,cxq^i,(c/x)q^i;p)} 
=\frac{(bc,c/b,ax,a/x;q,p)_k}{(ab,a/b,cx,c/x;q,p)_k},  
\nonumber\\*
& \quad\, \text{for}\quad k \in \{0, 1, \dots, m\}, 
\nonumber\\*
A(0, \ell) &= \prod_{j=0}^{\ell-1}
\frac{\theta(acq^j,(c/a)q^j,bxq^j,(b/x)q^j;p)}
{\theta(abq^j,(b/a)q^j,cxq^j,(c/x)q^j;p)}
=\frac{(ac,c/a,bx,b/x; q,p)_{\ell}}{(ab,b/a,cx,c/x; q,p)_{\ell}}, 
\nonumber\\*
& \quad\,\text{for}\quad \ell \in \{0, 1, \dots, n \},
\nonumber\\*
A(k, \ell) 
&=B(k, \ell)\frac{(bc,c/b,ax,a/x;q,p)_k}
{(ab,a/b,cx,c/x;q,p)_k}\frac{(ac,c/a,bx,b/x; q,p)_{\ell}}
{(ab,b/a,cx,c/x; q,p)_{\ell}},
\nonumber\\*
& \quad \,\text{for}\quad  
k \in \{1,2, \dots, m\} \quad\text{and}\quad 
\ell \in \{1,2, \dots, n\}.
\end{align*}
Combining the last equation with Lemma \ref{thm:Belliptic}, we have
\begin{align*}
A(k, \ell)
&=
\theta \Big(\frac{a}{b} q^{k-\ell};p \Big)
\frac{(bcq^{\ell},c/b,ax,a/x;q, p)_k 
(q^{k+1}, acq^k,c/a,bx,b/x; q, p)_{\ell}}
{(a/b; q,p)_{k+1} (q, q(b/a);q, p)_{\ell}(ab, cx, c/x; q, p)_{k+\ell}} q^{\ell},
\nonumber\\*
&=
\theta \Big(\frac{b}{a} q^{\ell-k};p \Big)
\frac{(acq^k,c/a,bx,b/x;q,p)_{\ell} 
(q^{\ell+1}, bcq^{\ell}, c/b, ax, a/x; q, p)_{k}}
{(b/a; q,p)_{\ell+1} (q, q(a/b);q,p)_{k}(ab,cx,c/x;q,p)_{\ell+k}} q^{k},
\nonumber\\*
& \quad \text{for}\quad (k, \ell) \in \Lambda_{m,n}. 
\end{align*}
Inserting this explicit expression for $A(k,l)$,
and those for $h(i,j)$ and $1-h(i,j)$ from \eqref{h_elliptic1} 
and \eqref{h_symmetry}, respectively, into
\eqref{equality} settles \eqref{CB_elliptic1}.
\end{proof}

\section{Chaundy--Bullard type identities viewed
  as B\'ezout identities} 
\label{sec:Bezout}

As pointed out in \cite[Remark~2.2]{KS08}, 
the Chaundy--Bullard identity \eqref{CB1} 
can be regarded as a \textit{B\'ezout identity},
\begin{equation}
1=P^{(1)}(x)Q^{(1)}(x)+P^{(2)}(x)Q^{(2)}(x).
\label{Bezout1}
\end{equation}
Here $P^{(1)}(x)$ and $P^{(2)}(x)$ are polynomials in $\C[x]$
of degrees $n+1$ and $m+1$, respectively,  
with no common zeros. 
The polynomials $Q^{(1)}(x)$ and $Q^{(2)}(x)$ in 
\eqref{Bezout1} have degree $m$ and $n$, respectively, and 
are moreover unique.
If we set $P^{(1)}(x)=(1-x)^{n+1}$ and $P^{(2)}(x)=x^{m+1}$, 
they are polynomials without common zeros of degree
$n+1$ and $m+1$, respectively.
Hence we have the equality
\begin{equation}
1=(1-x)^{n+1} Q^{(1)}_{m,n}(x)+x^{m+1} Q^{(2)}_{m,n}(x),
\label{Bezout12}
\end{equation}
where $Q^{(1)}_{m,n}(x)$ and $Q^{(2)}_{m,n}(x)$ are
polynomials of degree $m$ and $n$, respectively.
They are uniquely determined by using the symmetry
$Q^{(2)}_{m,n}(x)=Q^{(1)}_{n,m}(1-x)$ and finding
$Q^{(1)}_{m,n}(x)$ by dividing both sides of the identity
in \eqref{Bezout12} by $(1-x)^{n+1}$ and carrying out
Taylor expansion in $x$,
with the result being given by \eqref{CB1}. 

It is natural to ask whether this interpretation extends
and can be used to also prove the various extensions of the
Chaundy--Bullard identity. This can indeed be done for
three of the basic extensions (namely, the $q$-extension
and the two $(a,b;q)$-extensions), as these implicitly involve
polynomial bases. We did not succeed in interpreting
the $(a,b,c;q)$-extension or the elliptic extension of the
Chaundy--Bullard identity as B\'ezout identities
(as the underlying bases there are not polynomial bases but
special rational function bases, respectively, elliptic
function bases).

The analysis in this section
strongly connects to results in basic hypergeometric series
(cf.\ \cite{GR04}).
The \textit{basic hypergeometric series} $_{r+1}\phi_r$ is
defined by
\[
{}_{r+1}\phi_r\left[\begin{matrix}a_0, a_1, \dots, a_r\\
b_1, \dots, b_r \end{matrix};q, z\right]
:= \sum_{k=0}^{\infty}
\frac{(a_0,a_1,\cdots,a_r; q)_k}
{(q,b_1,\cdots,b_r; q)_k} z^k,
\]
where $|z|<1$ and $|q|<1$, if the series does not terminate,
for convergence \cite{GR04}. 

\subsection{$q$-Extension of the Chaundy--Bullard identity}
\label{sec:qCB_B}


While the first and second term on the right-hand side of
original Chaundy--Bullard identity 
\eqref{CB1}
are evidently symmetric with respect to the simple involution
$(x,n,m)\mapsto (1-x,m,n)$,
the corresponding symmetry for \eqref{CB_q1}, i.e.
\begin{equation}
1=(x; q)_{n+1} \sum_{k=0}^{m}
\begin{bmatrix}n+k\\k \end{bmatrix}_{q}
x^k
+x^{m+1} \sum_{k=0}^n 
\begin{bmatrix}m+k\\k \end{bmatrix}_{q}
q^k (x; q)_k,
\label{CB_q12}
\end{equation}
is less evident.
In order to clarify the ``hidden'' symmetry, 
we shall look at the
transition matrix of the two respective polynomial bases
$[x^n]_{n \in \N_0}$ and $[(x;q)_n]_{n \in \N_0}$.
This will enable us to interpret \eqref{CB_q12}
as a B\'ezout identity \eqref{Bezout1}.

By the $q$-binomial theorem, we have (cf.\ \cite[Ex.~1.2 (vi)]{GR04})
\begin{equation*}
(x;q)_n=\sum_{k=0}^n
\begin{bmatrix}n\\k\end{bmatrix}_q
(-1)^kq^{\binom k2}x^k.
\end{equation*}
Now, defining the lower-triangular
matrix $F=(f_{nk})_{n,k\in \N_0}$ by its entries
\begin{equation*}
f_{nk}=\begin{bmatrix}n\\k\end{bmatrix}_q(-1)^kq^{\binom k2},
\end{equation*}
the inverse of $F$ is known to be the lower-triangular
matrix $G=(g_{nk})_{n,k\in \N_0}$ with entries
\begin{equation*}
g_{nk}=\begin{bmatrix}n\\k\end{bmatrix}_q(-1)^kq^{\binom k2+k(1-n)}.
\end{equation*}
(This matrix inversion is equivalent to the case $a\to 0$ in
the matrix inversion for $B$ in \cite{B83}.)
A simple computation reveals that
\begin{equation*}
g_{nk}(q)=f_{nk}(q^{-1}).
\end{equation*}
Therefore, the relation
\begin{subequations}\label{invrel}
\begin{equation}
(x;q)_n=\sum_{k=0}^nf_{nk}(q)x^k
\end{equation}
is equivalent to
\begin{equation}
x^n=\sum_{k=0}^nf_{nk}(q^{-1})(x;q)_k.
\end{equation}
\end{subequations}

Let $\C(q)[x]$ be the vector space of polynomials in $x$ 
with coefficients that are rational functions in $q$ 
over the field $\C$. 
We define the linear operator $T$ on $\C(q)[x]$ by
\[
T\,\sum_{k\ge 0}c_k(q)x^k=\sum_{k\ge 0}c_k(q^{-1})(x;q)_k.
\]
Note that $T$ is an {\em involution} (this follows immediately
from \eqref{invrel}) but not a homomorphism (unless $q=1$).
To derive \eqref{CB_q1} using B\'ezout's identity, observe
that for each $m,n$ there exist unique polynomials
$Q^{(1)}_{m,n}$ and $Q^{(2)}_{m,n}$ of degree $m,n$, respectively, 
such that
\begin{equation*}
1=(x;q)_{n+1} Q^{(1)}_{m,n}(x;q)+x^{m+1} Q^{(2)}_{m,n}(x;q).
\end{equation*}
This implies
\begin{equation*}
\frac 1{(x;q)_{n+1}}=Q^{(1)}_{m,n}(x;q)+
\cO(x^{m+1})\qquad \text{as $x\to 0$.}
\end{equation*}
The next step is to compute $Q^{(1)}_{m,n}(x;q)$ by using
\begin{equation}
\label{fid}
\frac 1{(x;q)_{n+1}}=\sum_{k=0}^m
\begin{bmatrix}n+k\\k\end{bmatrix}_q x^k+
\cO(x^{m+1})\qquad \text{as $x\to 0$,}
\end{equation}
a result which is easily deduced from 
the non-terminating $q$-binomial theorem \cite[Equation (II.3)]{GR04}.

If we can show that
\begin{equation*}
T\bigg((x;q)_{n+1}\sum_{k=0}^m
\begin{bmatrix}n+k\\k\end{bmatrix}_q x^k\bigg)
=x^{n+1} Q^{(2)}_{n,m}(x;q),
\end{equation*}
for
\begin{equation*}
Q^{(2)}_{m,n}(x;q)
=\sum_{k=0}^n\begin{bmatrix}m+k\\k\end{bmatrix}_q q^k(x;q)_k,
\end{equation*}
then we are done,
as $Q^{(2)}_{n,m}(x;q)$ must have degree $m$ and be unique.
The computations are as follows:
\begin{align*}
&T\bigg((x;q)_{n+1}\sum_{k=0}^m
\begin{bmatrix}n+k\\k\end{bmatrix}_q x^k\bigg)\\
&=T \sum_{\ell=0}^{n+1}\sum_{k=0}^m
\begin{bmatrix}n+1\\ \ell\end{bmatrix}_q
(-1)^{\ell} q^{\binom \ell 2}
\begin{bmatrix}n+k\\k\end{bmatrix}_q x^{k+\ell}
\\
&=\sum_{\ell=0}^{n+1}\sum_{k=0}^m
\begin{bmatrix}n+1\\ \ell \end{bmatrix}_{q^{-1}}
(-1)^\ell q^{-\binom \ell 2}
\begin{bmatrix}n+k\\k\end{bmatrix}_{q^{-1}}(x;q)_{k+\ell}
\\
&=\sum_{k=0}^m\begin{bmatrix}n+k\\k\end{bmatrix}_{q^{-1}}(x;q)_k
\sum_{\ell=0}^{n+1}
\begin{bmatrix}n+1\\ \ell\end{bmatrix}_{q^{-1}}
(-1)^{\ell} q^{-\binom \ell 2}(xq^k;q)_{\ell}
\\
&=\sum_{k=0}^m\begin{bmatrix}n+k\\k\end{bmatrix}_{q^{-1}}(x;q)_k\,
x^{n+1}q^{k(n+1)}
\\
&=x^{n+1}\sum_{k=0}^m
\begin{bmatrix}n+k\\k\end{bmatrix}_qq^k(x;q)_k,
\end{align*}
which settles \eqref{CB_q12}.

\subsection{$(a,b;q)$-extension of 
the first kind of the Chaundy--Bullard identity}
\label{sec:abq1CB_B}

Next we consider the $(a,b;q)$-extensions of the fist kind; 
\begin{equation}\label{aqdjeq}
1=\frac{(bx;q)_{n+1}}{(b/a;q)_{n+1}}
\sum_{k=0}^m\frac{(q^{n+1},ax;q)_k}{(q,aq/b;q)_k}q^k+
\frac{(ax;q)_{m+1}}{(a/b;q)_{m+1}}
\sum_{k=0}^n\frac{(q^{m+1},bx;q)_k}{(q,bq/a;q)_k}q^k.
\end{equation}
The transition matrix 
$F=(f_{nk})_{n,k\in \N_0}$
between the polynomial bases
$[(ax;q)_n]_{n \in \N_0}$ and $[(bx;q)_n]_{n \in \N_0}$
is clearly symmetric in $a$ and $b$.
We actually do not need its explicit form here but nevertheless
note that the connection coefficients are given by
\begin{equation*}
f_{nk}(a,b;q)=\frac{(b/a;q)_n (q^{-n};q)_k}{(q,aq^{1-n}/b;q)_k}q^k
\end{equation*}
(the coefficients of the inverse sequence are
$g_{nk}(a,b;q)=f_{nk}(b,a;q)$); the connecting relation
\begin{equation*}
\sum_{k=0}^nf_{nk}(a,b;q)(ax;q)_k=(bx;q)_n
\end{equation*}
is equivalent to the $q$-Chu--Vandermonde summation
\cite[Equation~(II.6)]{GR04} (which conversely uniquely
determines the connection coefficients $f_{nk}(a,b;q)$).

To prove \eqref{aqdjeq} using B\'ezout's identity, observe
that for each $m,n$ there exist unique polynomials
$Q^{(1)}_{m,n}$ and $Q^{(2)}_{m,n}$ of degree $m,n$, respectively, such that
\begin{equation}\label{abid}
1=(bx;q)_{n+1}Q^{(1)}_{m,n}(a,b,x;q)+(ax;q)_{m+1}Q^{(2)}_{m,n}(a,b,x;q).
\end{equation}
This implies
\begin{equation*}
\frac 1{(bx;q)_{n+1}}=Q^{(1)}_{m,n}(a,b,x;q)\qquad
 \mod (ax;q)_{m+1}.
\end{equation*}
The next step is to compute $Q^{(1)}_{m,n}(a,b,x;q)$ by using
the $(a,b,c)\mapsto(q^{n+1},ax,aq/b)$
special case of \cite[Appendix~(II.23)]{GR04}, which is
\begin{equation*}
\frac {(b/a;q)_{n+1}}{(bx;q)_{n+1}}=
{}_2\phi_1\left[\begin{matrix}q^{n+1},ax\\aq/b\end{matrix};q,q\right]
+\frac{(b/a,q^{n+1},ax;q)_\infty}{(a/b,bq^{n+1}/a,bx;q)_\infty}\,
{}_2\phi_1\left[\begin{matrix}bq^{n+1}/a,bx\\bq/a\end{matrix};q,q\right].
\end{equation*}
This implies
\begin{equation*}
\frac 1{(bx;q)_{n+1}}=\frac 1{(b/a;q)_{n+1}}\,
\sum_{k=0}^m\frac{(q^{n+1},ax;q)_k}{(q,aq/b;q)_k}q^k
\qquad \mod (ax;q)_{m+1},
\end{equation*}
for $(ax;q)_{m+1}$ divides $(ax;q)_k$ for each $k>m$.
Hence we deduce
\begin{equation*}
Q^{(1)}_{m,n}(a,b,x;q)=\frac 1{(b/a;q)_{n+1}}\,
\sum_{k=0}^m\frac{(q^{n+1},ax;q)_k}{(q,aq/b;q)_k}q^k.
\end{equation*}

Since, in addition to \eqref{abid}, we also have
\begin{equation*}
1=(bx;q)_{n+1}Q^{(2)}_{n,m}(b,a,x;q)+(ax;q)_{m+1}Q^{(1)}_{n,m}(b,a,x;q),
\end{equation*}
it follows by uniqueness ($Q^{(1)}_{m,n}(a,b,x;q)$ and $Q^{(2)}_{n,m}(b,a,x;q)$
have the same degree $m$) that we must have
\begin{equation*}
Q^{(1)}_{m,n}(a,b,x;q)=Q^{(2)}_{n,m}(b,a,x;q),
\end{equation*}
which settles \eqref{aqdjeq}.

\subsection{$(a,b;q)$-extension of the second kind of the Chaundy--Bullard identity}
\label{sec:abq2CB_B}

The following identity of polynomials in $\C(q,a,b)[x+x^{-1}]$
is the $(a,b;q)$-extension of the second kind,
\begin{align}\label{abqdjeq}
1&=\frac{(bx,b/x;q)_{n+1}}{(ab,b/a;q)_{n+1}}
   \sum_{k=0}^m\frac{(q^{n+1},ax,a/x;q)_k}{(q,aq/b,abq^{1+n};q)_k}q^k\\*
  \nonumber
&\quad+\frac{(ax,a/x;q)_{m+1}}{(ab,a/b;q)_{m+1}}
\sum_{k=0}^n\frac{(q^{m+1},bx,b/x;q)_k}{(q,bq/a,abq^{1+m};q)_k}q^k.
\end{align}
The transition matrix between the two polynomial sequences
$[(ax,a/x;q)_n]_{n \in \N_0}$ and $[(bx,b/x;q)_n]_{n \in \N_0}$
is clearly symmetric in $a$ and $b$.
We actually do not need its explicit form here but nevertheless
note that the connection coefficients are given by
\begin{equation*}
  \widetilde
  f_{nk}(a,b;q)=\frac{(ab,b/a;q)_n (q^{-n};q)_k}{(q,ab,aq^{1-n}/b;q)_k}q^k
\end{equation*}
(the coefficients of the inverse sequence are
$ \widetilde
g_{nk}(a,b;q)=f_{nk}(b,a;q)$); the connecting relation
\begin{equation*}
  \sum_{k=0}^n \widetilde
  f_{nk}(a,b;q)(ax,a/x;q)_k=(bx,b/x;q)_n
\end{equation*}
is equivalent to the $q$-Pfaff--Saalsch\"utz summation
\cite[Equation~(II.12)]{GR04} (which conversely uniquely
determines the connection coefficients $ \widetilde
f_{nk}(a,b;q)$).

To prove \eqref{abqdjeq} using B\'ezout's identity, observe
that for each $m,n$ there exist unique polynomials
$Q^{(1)}_{m,n}$ and $Q^{(2)}_{m,n}$ of degree $m,n$, respectively, such that
\begin{equation}\label{abbid}
1=(bx,b/x;q)_{n+1}Q^{(1)}_{m,n}(a,b,x+x^{-1};q)+
(ax,a/x;q)_{m+1}Q^{(2)}_{m,n}(a,b,x+x^{-1};q).
\end{equation}
This implies
\begin{equation*}
\frac 1{(bx,b/x;q)_{n+1}}=Q^{(1)}_{m,n}(a,b,x+x^{-1};q)\qquad
\mod (ax,a/x;q)_{m+1}.
\end{equation*}
The next step is to compute $Q^{(1)}_{m,n}(a,b,x+x^{-1};q)$ by using
the $(a,b,c,e,f)\mapsto (q^{n+1},ax,a/x,aq/b,abq^{n+1})$
special case of \cite[Appendix~(II.24)]{GR04}, which is
\begin{align*}
\frac {(ab,b/a;q)_{n+1}}{(bx,b/x;q)_{n+1}}={}&{}
{}_3\phi_2\left[\begin{matrix}q^{n+1},ax,a/x\\
aq/b,abq^{n+1}\end{matrix};q,q\right]\\*
&{}+\frac{(b/a,q^{n+1},ax,a/x,b^2q^{n+1};q)_\infty}
{(a/b,bq^{n+1}/a,bx,b/x,abq^{n+1};q)_\infty}\,
{}_3\phi_2\left[\begin{matrix}bq^{n+1}/a,bx,b/x\\
bq/a,b^2q^{n+1}\end{matrix};q,q\right].
\end{align*}
The last equation implies
\begin{align*}
&\frac 1{(bx,b/x;q)_{n+1}}\\&=\frac 1{(ab,b/a;q)_{n+1}}\,
\sum_{k=0}^m\frac{(q^{n+1},ax,a/x;q)_k}{(q,aq/b,abq^{n+1};q)_k}q^k
\qquad \mod (ax,a/x;q)_{m+1},
\end{align*}
for $(ax,a/x;q)_{m+1}$ divides $(ax,a/x;q)_k$ for each $k>m$.
Hence we deduce
\begin{equation*}
Q^{(1)}_{m,n}(a,b,x+x^{-1};q)=\frac 1{(ab,b/a;q)_{n+1}}\,
\sum_{k=0}^m\frac{(q^{n+1},ax,a/x;q)_k}{(q,aq/b,abq^{n+1};q)_k}q^k.
\end{equation*}

Since, in addition to \eqref{abbid}, we also have
\begin{equation*}
1=(bx,b/x;q)_{n+1}Q^{(2)}_{n,m}(b,a,x+x^{-1};q)+
(ax,a/x;q)_{m+1}Q^{(1)}_{n,m}(b,a,x+x^{-1};q),
\end{equation*}
it follows by uniqueness ($Q^{(1)}_{m,n}(a,b,x+x^{-1};q)$ and
$Q^{(2)}_{n,m}(b,a,x+x^{-1};q)$
have the same degree $m$) that we must have
\begin{equation*}
Q^{(1)}_{m,n}(a,b,x+x^{-1};q)=Q^{(2)}_{n,m}(b,a,x+x^{-1};q),
\end{equation*}
which settles \eqref{abqdjeq}.

\section{Variants of the $q$-extended Chaundy--Bullard 
identities and $q$-commuting variables}
\label{sec:q_commuting}

In the original \eqref{CB1}, if we replace $x$ by $x/(x-1)$
and multiply the identity by $(1-x)^{m+n+1}$,
a variant is obtained, namely
\begin{align}\label{CB2}
(1-x)^{m+n+1} &=\sum_{k=0}^{m}
\binom{n+k}{k} (-1)^k x^k (1-x)^{m-k}\\*\nonumber
& \quad
+(-1)^{m+1} x^{m+1} \sum_{k=0}^n 
\binom{m+k}{k} (1-x)^{n-k}.
\end{align}
As discussed in \cite{KS08}, 
\eqref{CB1} is equivalent to the identity in the form
involving two variables $x$ and $y$, 
\begin{align}
\label{CBxy1}
(x+y)^{m+n+1}&=y^{n+1} \sum_{k=0}^{m}
\binom{n+k}{k} x^k (x+y)^{m-k}\\*\nonumber
&\quad+x^{m+1} \sum_{k=0}^n 
\binom{m+k}{k} y^k (x+y)^{n-k},
\end{align}
which is in \textit{homogeneous form}. 

In order to give a $q$-extension of the
homogeneous Chaundy--Bullard identity 
\eqref{CBxy1}, 
we consider the unital algebra 
$\C_q[X, Y]$ defined over
$\C$ generated by $X, Y$, satisfying the relation
\[
YX = q XY.
\]
$\C_q[X,Y]$ can be regarded as a $q$-deformation of the
commutative algebra $\C[x,y]$.
We call $X,Y$ forming $\C_q[X,Y]$ 
\textit{$q$-commuting variables}.
The following \textit{binomial theorem for 
$q$-commuting variables} is well known 
(see \cite{Koo97,Sch20} and references therein),
\[
(X+Y)^n=\sum_{k=0}^n 
\begin{bmatrix}n\\k\end{bmatrix}_q X^k Y^{n-k},
\]
for $X, Y \in \C_q[X,Y]$,
where the $q$-binomial coefficient is defined by
\eqref{q_binom1}. 

Then, from Theorem~\ref{ecbthm} we obtain that
a $q$-extension of \eqref{CBxy1} is given by 
\begin{align*}
(X+Y)^{m+n+1} &= Y^{n+1} \sum_{k=0}^{m}
\begin{bmatrix}n+k\\k\end{bmatrix}_q
q^{-(n+1)k}\,X^k (X+Y)^{m-k}\\*
& \quad
+
X^{m+1} \sum_{k=0}^n 
\begin{bmatrix}m+k\\k\end{bmatrix}_{q^{-1}}
q^{(m+1)k}\, Y^k (X+Y)^{n-k},
\end{align*}
for $X, Y \in \C_q[X,Y]$.
This is the homogeneous form of the
$q$-extension of the Chaundy--Bullard identity 
\eqref{CB_q1}. 
The left-hand and right-hand sides of this identity are
both invariant under the transformation
$(X, Y, q, m, n) \mapsto (Y, X, q^{-1}, n, m)$.


\section{Elliptic extensions of the binomial theorem}
\label{sec:elliptic_binomial}
Recall (from Section~\ref{sec:introduction}) that we
denote by $\mathbb E_{a_1,a_2,\ldots,a_s;q,p}$ the field of totally
elliptic functions over $\C$, in the complex variables 
$\log_qa_1,\ldots,\log_qa_s$, with equal
periods $\sigma^{-1}$, $\tau\sigma^{-1}$
(where $q=e^{2\pi\sqrt{-1}\sigma}$, $p=e^{2\pi\sqrt{-1}\tau}$,
$\sigma,\tau\in\C$, $\Im\tau>0$), of double periodicity.

\subsection{An elliptic extension of the binomial theorem}
\label{subsec:elliptic_binomial}
In \cite[Theorem~2]{Sch20} one of the authors proved an
elliptic extension of the binomial theorem which we recall for convenience
and for comparison with our new analogous result in Theorem~\ref{vwdbinth}
below.

For indeterminates $a$, $b$, complex numbers $q$, $p$ (with $|p|<1$),
and nonnegative integers $n$, $k$,
define a variant of elliptic binomial coefficients, which here,
for better distinction from \eqref{eq:HBin},
we may refer to as \textit{$W$-binomial coefficients}
as follows
(this is exactly the expression for $w(\mathcal P((0,0)\to(k,n-k)))$
in \cite[Theorem~2.1]{Sch07}):
\begin{equation}\label{ellbc}
\begin{bmatrix}n\\k\end{bmatrix}_{a,b;q,p}:=
\frac{(q^{1+k},aq^{1+k},bq^{1+k},aq^{1-k}/b;q,p)_{n-k}}
{(q,aq,bq^{1+2k},aq/b;q,p)_{n-k}}.
\end{equation}
The $W$-binomial coefficient is indeed elliptic.
In particular,
$\left[\begin{smallmatrix}n\\k\end{smallmatrix}\right]_{a,b;q,p}
\in\E_{a,b,q^n,q^k;q,p}$.

It is immediate from the definition of \eqref{ellbc} that (for
integers $n,k$) there holds
\begin{subequations}\label{qbinrel}
\begin{equation}
\begin{bmatrix}n\\0\end{bmatrix}_{a,b;q,p}=
\begin{bmatrix}n\\n\end{bmatrix}_{a,b;q,p}=1,
\end{equation}
and
\begin{equation}
\begin{bmatrix}n\\k\end{bmatrix}_{a,b;q,p}=0,\qquad\text{whenever}\quad
k=-1,-2,\dots,\quad\text{or}\quad k> n.
\end{equation}
Furthermore, using the Weierstra{\ss}--Riemann addition
formula in \eqref{Weierstrass_Riemann}
one can verify the following recursion formula for the
elliptic binomial coefficients:
\begin{equation}\label{rec}
\begin{bmatrix}n+1\\k\end{bmatrix}_{a,b;q,p}=
\begin{bmatrix}n\\k\end{bmatrix}_{a,b;q,p}+
\begin{bmatrix}n\\k-1\end{bmatrix}_{a,b;q,p}\,W_{a,b;q,p}(k,n+1-k),
\end{equation}
\end{subequations}
for nonnegative integers $n$ and $k$, where the
elliptic weight function $W_{a,b;q,p}$ is defined on $\N_0^2$ as
\begin{equation}\label{Wdef}
W_{a,b;q,p}(s,t):=\frac{\ta(aq^{s+2t},bq^{2s},bq^{2s-1},aq^{1-s}/b,aq^{-s}/b;p)}
{\ta(aq^s,bq^{2s+t},bq^{2s+t-1},aq^{1+t-s}/b,aq^{t-s}/b;p)}q^t.
\end{equation}

Clearly, $W_{a,b;q,p}(s,0)=1$, for all $s$.
If one lets $p\to 0$, $a\to 0$, then $b\to 0$ (in this order), the weights
in \eqref{Wdef} reduce to the standard $q$-weights
$$W_q(s,t):=\lim_{b\to 0}\Big(\lim_{a\to 0}
\Big(\lim_{p\to 0}W_{a,b;q,p}(s,t)\Big)\Big)=q^t,$$
and the
relations in \eqref{qbinrel} reduce to
\begin{equation*}
\begin{bmatrix}n\\0\end{bmatrix}_{q}=
\begin{bmatrix}n\\n\end{bmatrix}_{q}=1,
\end{equation*}
\begin{equation*}
\begin{bmatrix}n+1\\k\end{bmatrix}_{q}=
\begin{bmatrix}n\\k\end{bmatrix}_{q}+
\begin{bmatrix}n\\k-1\end{bmatrix}_{q}\,q^{n+1-k},
\end{equation*}
for positive integers $n$ and $k$ with $n\ge k$, which is
a well-known recursion for the $q$-binomial coefficients.
(If instead, one lets $p\to 0$, $b\to 0$, and then $a\to 0$ (in this order),
the weight function in \eqref{Wdef} reduces to $q^{-t}$ and
 \eqref{qbinrel} reduces to the recursion for the $q$-binomial
 coefficients where $q$ has been replaced by $q^{-1}$.)
In \cite{Sch07} lattice paths in the integer
lattice $\N_0^2$ were enumerated with respect to precisely this
weight function. A similar weight function was subsequently
used by Borodin, Gorin and Rains in \cite[Section~10]{BGR10}
(see in particular the expression obtained for
$\frac{w(i,j+1)}{w(i,j)}$ on p.~780 of that paper)
in the context of weighted lozenge tilings. 

\begin{definition}\label{defea}
For two complex numbers $q$ and $p$ with $|p|<1$,
let\break $\C_{q,p}[X,Y,\E_{a,b;q,p}]$ denote
the associative unital algebra over $\C$,
generated by $X$, $Y$, and the commutative
subalgebra $\E_{a,b;q,p}$,
satisfying the following three relations:
\begin{subequations}\label{defeaeq}
\begin{align}
YX&=W_{a,b;q,p}(1,1)\,XY,\label{elleq}\\
Xf(a,b)&=f(aq,bq^2)\,X,\label{xf}\\\label{yf}
Yf(a,b)&=f(aq^2,bq)\,Y,
\end{align}
\end{subequations}
for all $f\in\E_{a,b;q,p}$.
\end{definition} 
We refer to the variables $X,Y,a,b$
forming $\C_{q,p}[X,Y,\E_{a,b;q,p}]$
as {\em elliptic commuting} variables.
The algebra $\C_{q,p}[X,Y,\E_{a,b;q,p}]$ reduces to
$\C_{q}[X,Y]$ if one formally lets $p\to 0$, $a\to 0$,
then $b\to 0$ (in this order), while (having eliminated the nome $p$) 
relaxing the condition of ellipticity.
It should be noted that the monomials $X^kY^l$ form a basis for the
algebra $\C_{q,p}[X,Y,\E_{a,b;q,p}]$ as a left module over $\E_{a,b;q,p}$,
i.e., any element can be written uniquely as a finite sum
$\sum_{k,l\ge 0} f_{kl}X^kY^l$ with $f_{kl}\in \E_{a,b;q,p}$ which
we call the \textit{normal form} of the element.

The following result from \cite[Theorem~2]{Sch20}
shows that the normal form of the binomial
$(X+Y)^n$ is nice; each (left) coefficient of $X^kY^{n-k}$ completely
factorizes as an expession in $\E_{a,b;q,p}$.

\begin{theorem}[Binomial theorem for 
  variables in \mbox{$\C_{q,p}[X,Y,\E_{a,b;q,p}]$}]\label{ebthm}
  Let $n\in\N_0$.
  Then, as an identity in  $\C_{q,p}[X,Y,\E_{a,b;q,p}]$, we have
\begin{equation}\label{eqbinth}
(X+Y)^n=\sum_{k=0}^n\begin{bmatrix}n\\k\end{bmatrix}_{a,b;q,p}X^kY^{n-k}.
\end{equation}
\end{theorem}

In \cite{Sch20} convolution was applied to this result
(together with comparison of coefficients) to
recover Frenkel and Turaev's ${}_{10}V_9$ summation \cite{FT97}
(see also \cite[Equation~(11.4.1)]{GR04}),
an identity which is fundamental to the
theory of elliptic hypergeometric series:

\begin{proposition}[Frenkel and Turaev's ${}_{10}V_9$ summation]\label{propft}
Let $n\in\N_0$ and $a,b,c,d,e,q,p\in\C$ with $|p|<1$.
Then there holds the following identity:
\begin{align}\label{propfteq}
\sum_{k=0}^n\frac{\ta(aq^{2k};p)}{\ta(a;p)}
\frac{(a,b,c,d,e,q^{-n};q,p)_k}
{(q,aq/b,aq/c,aq/d,aq/e,aq^{n+1};q,p)_k}q^k&\\\nonumber
=
\frac{(aq,aq/bc,aq/bd,aq/cd;q,p)_n}
{(aq/b,aq/c,aq/d,aq/bcd;q,p)_n}&,
\end{align}
where $a^2q^{n+1}=bcde$.
\end{proposition}

It is straightforward to use the lattice path model
to derive a homogeneous Chaundy--Bullard identity
for the elliptic commuting variables forming
the algebra  $\C_{q,p}[X,Y,\E_{a,b;q,p}]$. (We omit the details;
it is also principle possible to verify the identity by induction.)
The result is as follows:
\begin{theorem}[Homogeneous Chaundy--Bullard identity for 
  variables in\break 
  \mbox{$\C_{q,p}[X,Y,\E_{a,b;q,p}]$}]\label{ecbthm}
  Let $n,m\in\N_0$. Then the following identity is valid in\break 
  $\C_{q,p}[X,Y,\E_{a,b;q,p}]$:
\begin{align}\label{eq:ecbthm}
  (X+Y)^{m+n+1}&=\sum_{k=0}^m\begin{bmatrix}n+k\\k\end{bmatrix}_{a,b;q,p}
  X^kY^{n+1}(X+Y)^{m-k}\\*\nonumber
  &\quad+\sum_{k=0}^n\begin{bmatrix}m+k\\m\end{bmatrix}_{a,b;q,p}
  W_{a,b;q,p}(m+1,k)\,X^{m+1}Y^k(X+Y)^{n-k}.
\end{align}
\end{theorem}

\subsection{A new elliptic extension of the binomial theorem}
\label{subsec:new_elliptic_binomial}
Inspired by the lattice path model in Section~\ref{sec:proof} of this
paper, we present a new elliptic extension of the binomial theorem
which is similar to Theorem~\ref{ebthm} but different, 
see Theorem~\ref{vwdbinth} below.

First define $h_{x; a,b,c;q,p}(i,j)$ as in \eqref{h_elliptic1}, and let
\begin{equation}\label{H_elliptic1}
H_{x; a, b, c: q, p}(i, j)
:= \frac{ h_{x;a,b,c;q,p}(i, j)}{h_{x;a,b,c;q,p}(i, 0)}.
\end{equation}
With $B(k, \ell)$ given by
\eqref{ellipticB} in Lemma~\ref{thm:Belliptic}, 
for $n \in \N_0$ we define a variant of elliptic binomial
coefficients, which (in view of Proposition~\ref{thm:H_binomial})
one may refer to as \textit{$H$-binomial coefficients},
by
\begin{equation}\label{eq:HBin}
\begin{bmatrix} n \\ k \end{bmatrix}_{x;a,b,c;q,p}
:=\begin{cases}
B(k, n-k), & \quad  k \in \{0, 1, \dots, n\},
\cr
0, & \quad \mbox{$k \in - \N_0$ or $k >n$}.
\end{cases}
\end{equation}
Then using the symmetry \eqref{h_symmetry}
in Lemma \ref{thm:h_symmetry}, 
the assertion of
Lemma \ref{thm:Belliptic} is rewritten as follows.

\begin{proposition}
\label{thm:H_binomial}
For $n \in \N_0$, $0 \leq k \leq n$, 
\begin{align}\label{binomialH}
\begin{bmatrix}n+1 \\ k \end{bmatrix}_{x; a,b,c;q,p}
&=
\begin{bmatrix}n \\ k-1 \end{bmatrix}_{x; a,b,c;q,p}
H_{x; a, b, c; q, p}(k-1, n+1-k)\\*\nonumber
& \quad 
+
\begin{bmatrix}n \\ k \end{bmatrix}_{x; a,b,c;q,p}
H_{x; b, a, c; q, p}(n-k, k).
\end{align}
\end{proposition}

It is obvious that the $H$-binomial coefficients
$\begin{bmatrix}n \\ k \end{bmatrix}_{x; a,b,c;q,p}$ have a
nice combinatorial interpretation in terms of weighted lattice paths.
The generating function $w_{x; a,b,c;q,p}$ with respect to
the weights $H_{x; a,b,c;q,p}$ of all paths from $(0,0)$ to
$(k,n-k)$ is clearly
\begin{equation}\label{lpvwbc}
w_{x; a,b,c;q,p}(\mathcal P((0,0)\to (k,n-k)))=
\begin{bmatrix}n \\ k \end{bmatrix}_{x; a,b,c;q,p}.
\end{equation}

We now define a new elliptic extension of the non-commutative
algebra $\C_q[X,Y]$:

\begin{definition}\label{defea2}
For two complex numbers $q$ and $p$ with $|p|<1$,
let\break $\C_{q,p}[X,Y,\E_{x,a,b,c;q,p}]$ denote
the associative unital algebra over $\C$,
generated by $X$, $Y$, and the commutative
subalgebra $\E_{x,a,b,c;q,p}$,
satisfying the following three relations:
\begin{subequations}\label{defea2eq}
\begin{align}
YX&=H_{x; a, b, c: q, p}(0, 1)\,XY\label{ell2eq}\\
Xf(x,a,b,c)&=f(x,aq,b,cq)\,X,\label{xf2}\\\label{yf2}
Yf(x,a,b,c)&=f(x,a,bq,cq)\,Y,
\end{align}
\end{subequations}
for all $f\in\mathbb E_{x,a,b,c;q,p}$.
\end{definition} 
Just as for the variables $X,Y,a,b$
forming $\C_{q,p}[X,Y,\E_{a,b;q,p}]$ in Definition~\ref{defea},
we refer to the variables $X,Y,x,a,b,c$ forming
$\C_{q,p}[X,Y,\E_{x,a,b,c;q,p}]$
in Definition~\ref{defea2}
as {\em elliptic-commuting} variables.

We note the following useful relations that follow from
\eqref{h_elliptic1} and \eqref{defea2eq}:
\begin{subequations}
\begin{align}
  X\,h_{x;a,b,c;q,p}(i,j)^{\pm1}
  &= h_{x;a,b,c;q,p}(i+1,j)^{\pm1}\,X,\\
  X\,h_{x;b,a,c;q,p}(j,i)^{\pm1}
  &=h_{x;b,a,c;q,p}(j,i+1)^{\pm1}\,X,\\
Y\,h_{x;a,b,c;q,p}(i,j)^{\pm1}&=h_{x;a,b,c;q,p}(i,j+1)^{\pm1}\,Y,\\
Y\,h_{x;b,a,c;q,p}(j,i)^{\pm1}&=h_{x;b,a,c;q,p}(j+1,i)^{\pm1}\,Y,
\end{align}
\end{subequations}
for all $(i,j)\in\N_0^2$.
Further, by induction on $r$ and $s$, one has
\begin{equation}\label{eq:YsXr}
Y^sX^r=\Big(\prod_{i=0}^{r-1}H_{x;a,b,c;q,p}(i,s)\Big)X^rY^s,
\end{equation}
for all $r, s \in \N_0$.

Similarly as in $\C_{q,p}[X,Y,\E_{a,b;q,p}]$, the monomials $X_kY^l$
form a basis for the algebra $\C_{q,p}[X,Y,\E_{x,a,b,c;q,p}]$, now as a
left module over $\E_{x,a,b,c;q,p}$. That is, any element can be
written uniquely as a finite sum
$\sum_{k,l\ge 0} f_{kl}X^kY^l$ with $f_{kl}\in \E_{x,a,b,c;q,p}$,
the normal form of the element.

The following non-commutative elliptic binomial theorem
which readily follows from \cite[Theorem~3]{Sch20}
(and could be independently proved by induction)
shows that the normal form of the binomial
$\big(X+ h_{x;b,a,c;q,p}(0,0)\,Y\big)^n$ is nice;
each (left) coefficient of $X^kY^{n-k}$ completely
factorizes in $\E_{x,a,b,c;q,p}$.

\begin{theorem}[Binomial theorem for elliptic commuting
  variables in\break \mbox{$\C_{q,p}[X,Y,\E_{x,a,b,c;q,p}]$}]\label{vwdbinth}
  Let $n\in\N_0$. Then, as an identity in
  $\C_{q,p}[X,Y,\E_{x,a,b,c;q,p}]$, we have
\begin{align}\label{vwdbinthid}
 &\big(X+ h_{x;b,a,c;q,p}(0,0)\,Y\big)^n\\*\nonumber
&  =\sum_{k=0}^n
\begin{bmatrix}n \\ k \end{bmatrix}_{x; a,b,c;q,p}\!
\bigg(\prod_{j=0}^{n-k-1} h_{x;b,a,c;q,p}(j,0)\bigg)
X^kY^{n-k}.
\end{align}
\end{theorem}

By convolution and comparison of coefficients
we obtain the following identity:
\begin{corollary}[An elliptic binomial convolution formula]\label{corftbinconv}
Let $n,m,k\in\N_0$ and $x,a,b,c,q,p\in\C$ with $|p|<1$.
Then there holds the following convolution
formula:
\begin{align}\label{corftbinconveq}
  &\begin{bmatrix}n+m\\k\end{bmatrix}_{x;a,b,c;q,p}
  \bigg(\prod_{j=0}^{n+m-k-1} h_{x;b,a,c;q,p}(j,0)\bigg)\\\nonumber
 & =
  \sum_{j=0}^k\Bigg(\begin{bmatrix}n\\j\end{bmatrix}_{x;a,b,c;q,p}
  \begin{bmatrix}m\\k-j\end{bmatrix}_{x;aq^j,bq^{n-j},cq^n;q,p}
  \bigg(\prod_{\ell=0}^{n-j-1} h_{x;b,a,c;q,p}(\ell,0)\bigg)\\*
  \nonumber
&\qquad\quad\times
  \bigg(\prod_{s=0}^{m+j-k-1} h_{x;b,a,c;q,p}(s+n-j,j)\bigg)
  \bigg(\prod_{i=0}^{k-j-1}H_{x;a,b,c;q,p}(i+j,n-j)\bigg)\Bigg).
\end{align}
\end{corollary}
\begin{proof}
  Working in $\C_{q,p}[X,Y,\E_{x,a,b,c;q,p}]$, one expands the
  binomial $(X+ h_{x;b,a,c;q,p}(0,0)\,Y)^{n+m}$
  in two different ways and suitably extracts coefficients.
On the one hand,
\begin{align}\label{eqwbinth1}
  &\big(X+h_{x;b,a,c;q,p}(0,0)\,Y\big)^{n+m}\\\nonumber
  &=\sum_{k=0}^{n+m}\begin{bmatrix}n+m\\k\end{bmatrix}_{x;a,b,c;q,p}
  \bigg(\prod_{j=0}^{n+m-1}h_{x;b,a,c;q,p}(j,0)\bigg)
X^kY^{n+m-k}.
\end{align}
On the other hand,
\begin{align}\label{eqwbinth2}
  &\big(X+h_{x;b,a,c;q,p}(0,0)\,Y\big)^{n+m}\\\nonumber
 & =\big(X+h_{x;b,a,c;q,p}(0,0)\,Y\big)^n\,
  \big(X+h_{x;b,a,c;q,p}(0,0)\,Y\big)^m\\*\notag
  &=\sum_{j=0}^n\sum_{r=0}^m\Bigg(
\begin{bmatrix}n\\j\end{bmatrix}_{x;a,b,c;q,p}
  \bigg(\prod_{\ell=0}^{n-j-1}h_{x;b,a,c;q,p}(\ell,0)\bigg)
  X^jY^{n-j}\\*\nonumber
  &\qquad\qquad\quad\times
    \begin{bmatrix}m\\r\end{bmatrix}_{x;a,b,c;q,p}
  \bigg(\prod_{s=0}^{m-r-1}h_{x;b,a,c;q,p}(s,0)\bigg)
  X^rY^{m-r}\Bigg)\\*\nonumber
  &=\sum_{j=0}^n\sum_{r=0}^m\Bigg(
\begin{bmatrix}n\\j\end{bmatrix}_{x;a,b,c;q,p}
    \begin{bmatrix}m\\r\end{bmatrix}_{x;aq^j,bq^{n-j},cq^n;q,p}
 \bigg(\prod_{\ell=0}^{n-j-1}h_{x;b,a,c;q,p}(\ell,0)\bigg)\\*\nonumber
  & \qquad\qquad\quad\times\bigg(\prod_{s=0}^{m-r-1}
    h_{x;b,a,c;q,p}(s+n-j,j)\bigg)
  X^jY^{n-j}X^rY^{m-r}\Bigg).
\end{align}
Now use \eqref{eq:YsXr} to apply
\begin{equation*}
X^jY^{n-j}X^rY^{m-r}
=\bigg(\prod_{i=0}^{r-1}H_{x;a,b,c;q,p}(i+j,n-j)\bigg)x^{j+r}y^{n+m-j-r}\quad
\text{for\/ $n\ge j$},
\end{equation*}
and extract and equate (left) coefficients of
$X^kY^{n+m-k}$ in \eqref{eqwbinth1} and \eqref{eqwbinth2}.
This gives the convolution formula \eqref{corftbinconveq}.
\end{proof}

\begin{remark}
Corollary~\ref{corftbinconv} is not a new result, but actually a
special case of the Frenkel--Turaev sum in \eqref{propfteq}.
Writing out all the expressions in \eqref{corftbinconveq} in
explicit terms using \eqref{h_elliptic1}, \eqref{ellipticB},
\eqref{H_elliptic1}, and \eqref{eq:HBin}, and applying some
elementary manipulations
(such as those appearing in \cite[p.~310]{GR04})
of the theta-shifted factorials, it becomes clear
that, up to a multiplicative factor that can be pulled out the sum,
the sum on the right-hand side of \eqref{corftbinconveq}
is indeed the  Frenkel--Turaev sum in \eqref{propfteq}
with respect to the substitutions
$$(a,b,c,d,e,n)\mapsto
(aq^{-n}/b,acq^{n+m},q^{1-n}/bc,aq^{k-n-m}/b,q^{-n},k).$$
(In particular, the factors depending on $x$ can all be pulled out
of the sum and can be canceled with those appearing on the
left-hand side. Also, the variable $b$ in the sum is redundant.)
By analytic continuation, the integers $n$
and $m$ can be replaced by continuous variables (say $\log_q\nu$
and $\log_q\mu$) and one recovers the full Frenkel--Tureav sum,
without the restriction of having parameters that are integer
powers of $q$.
\end{remark}

Finally, it is again straightforward to use the lattice path model
to derive a homogeneous Chaundy--Bullard identity
for the the elliptic commuting variables forming
the algebra $\C_{q,p}[X,Y,\E_{x,a,b,c;q,p}]$. (We omit the details;
it is also principle possible to verify the identity by induction.)
The result is as follows:
\begin{theorem}[Homogeneous Chaundy--Bullard identity for elliptic commuting
  variables in \mbox{$\C_{q,p}[X,Y,\E_{x,a,b,c;q,p}]$}]\label{ecbthm2}
  Let $n,m\in\N_0$.
Then the following identity is valid in $\C_{q,p}[X,Y,\E_{x,a,b,c;q,p}]$:
\begin{align}\label{eq:ecbthm2}
  &\big(X+h_{x;b,a,c;q,p}(0,0)\,Y\big)^{n+m+1}\\*\nonumber
  &=\sum_{k=0}^m\Bigg(\begin{bmatrix}n+k\\k\end{bmatrix}_{x;a,b,c;q,p}
  \bigg(\prod_{j=0}^{n-1} h_{x;b,a,c;q,p}(j,0)\bigg)\\*\nonumber
  &\qquad\qquad\times h_{x;b,a,c;q,p}(n,k)\,  X^kY^{n+1}
    \big(X+h_{x;b,a,c;q,p}(0,0)\,Y\big)^{m-k}\Bigg)\\*\nonumber
  &\quad+\sum_{k=0}^n\Bigg(\begin{bmatrix}m+k\\m\end{bmatrix}_{x;a,b,c;q,p}
  \bigg(\prod_{j=0}^{k-1} h_{x;b,a,c;q,p}(j,0)\bigg)\\*\nonumber
  &\qquad\qquad\;\times  
  X^{m+1}Y^k\big(X+h_{x;b,a,c;q,p}(0,0)\,Y\big)^{n-k}\Bigg).
\end{align}
\end{theorem}

\section*{Acknowledgements}
The authors wish to thank the referee for their helpful remarks.


\bibliographystyle{amsalpha}

\begin{thebibliography}{99}

\bibitem{BGR10}
A.~Borodin, V.~Gorin, E.~M.~Rains,
\emph{$q$-Distributions on boxed plane partitions},
Sel.\ Math.\ (N.S.) \textbf{16}(4) (2010), 731--789.

\bibitem{B83}
D.~M.~Bressoud,
\emph{A matrix inverse},
Proc.\ Amer.\ Math.\ Soc.\ \textbf{88}(3) (1983), 446--448.

\bibitem{CB60}
T.~W.~Chaundy and J.~E.~Bullard,
\emph{John Smith's problem},
Math.\ Gazette \textbf{44} (1960), 253--260.

\bibitem{FT97}
I.~B.~Frenkel and V.~G.~Turaev,
\emph{Elliptic solutions of the Yang--Baxter equation and modular
  hypergeometric functions},
in V.~I.~Arnold et al.\ (eds.),
{\em The Arnold--Gelfand Mathematical Seminars}, pp.~171--204,
Birkh\"auser, Boston, 1997.

\bibitem{GR04} 
G.~Gasper and M.~Rahman,
\emph{Basic hypergeometric series}, second edition,
Encyclopedia of Mathematics and Its Applications~96,
Cambridge University Press, Cambridge, 2004.


\bibitem{KM61}
P.~Kesava~Menon,
\emph{Summation of certain series},
J.\ Indian Math.\ Soc.\ (N.S.) \textbf{25} (1961), 121--128.

\bibitem{Koo97}
T.~H.~Koornwinder,
\emph{Special functions and $q$-commuting variables},
in M.~E.~H.~Ismail, D.~R.~Masson and M.~Rahman (eds.),
\emph{Special functions, $q$-series and related topics},
Fields Institute Communications \textbf{14} (1997), 131--166.
\texttt{arXiv:q-alg/9608008}.

\bibitem{Koo14}
T.~H.~Koornwinder,
\emph{On the equivalence of two fundamental theta identities},
Anal.\ Appl.\ (Singap.) \textbf{12} (2014), 711--725.

\bibitem{KS08}
T.~H.~Koornwinder and M.~J.~Schlosser,
\emph{On an identity by Chaundy and Bullard. I}, 
Indag.\ Math.\ (N.S.) \textbf{19}(2) (2008), 239--261.

\bibitem{KS13}
T.~H.~Koornwinder and M.~J.~Schlosser,
\emph{On an identity by Chaundy and Bullard. II.
  More history},
Indag.\ Math.\ (N.S.) \textbf{24} (2013), 174--180.

\bibitem{M09}
X.~R.~Ma,
\emph{Six-variable generalization of Ramanujan's reciprocity theorem},
J.\ Math.\ Anal.\ Appl.\ \textbf{353} (2009), 320--328.  

\bibitem{M1713}
P.~R.~de~Montmort,
{\em Essay d'analyse sur les jeux de hazard},
Seconde \'edition, Laurent le Conte, Paris, 1713; reprinted,
Chelsea, New York, 1980.

\bibitem{R21}
H.~Rosengren,
\emph{Elliptic hypergeometric functions},
in H.~S.~Cohl and M.~E.~H.~Ismail (eds.),
\emph{Lectures on Orthogonal Polynomials and Special Functions},
London Math. Soc. Lecture Note Ser.\ \text{464},
Cambridge University Press, 2021, pp. 213--279.
\texttt{arXiv:1608.06161v3}.

\bibitem{Sch07}
M.~J.~Schlosser,
\emph{Elliptic enumeration of nonintersecting lattice paths},
J.\ Combin.\ Theory Ser.\ A \textbf{114} (2007), 505--521.

\bibitem{Sch20}
M.~J.~Schlosser,
\emph{A noncommutative weight-dependent generalization 
of the binomial theorem}, 
S\'em.\ Lothar.\ Combin.\ \textbf{B81}j (2020), 24 pp.

\bibitem{Sh86}
B.\ Shiffman,
\emph{Complete characterization of holomorphic chains of codimension one},
Math.\ Ann.\ \textbf{274}(2) (1986), 233--256.


\bibitem{Sp02}
V.~P.~Spiridonov, 
\textit{Theta hypergeometric series},
in V.A.~Malyshev and A.M.~Vershik (eds.),
{\em Asymptotic Combinatorics with Applications to Mathematical Physics},
pp.~307--327, Kluwer Acad.\ Publ., Dordrecht, 2002;
\texttt{arXiv:math/0303204}.

\bibitem{Sp11}
V.~P.~Spiridonov, 
\textit{Theta hypergeometric terms},
in {\em Arithmetic and Galois theories of differential equations},
pp.~325--345, S\'emin.\ Congr.\ \textbf{23}, Soc.\ Math.\ France, Paris
2011;
\texttt{arXiv:math/1003.4491v3}.


\bibitem{WW62}
E.~T.~Whittaker and G.~N.~Watson,
\textit{A Course of Modern Analysis}, 4th ed.,
Cambridge University Press, Cambridge, 1962.

\end{thebibliography}


\end{document}